\documentclass[12pt, reqno]{amsart}
\usepackage{amsmath, amssymb, amsthm}

\usepackage{color}

\usepackage[colorlinks=true, urlcolor=blue, linkcolor=blue, citecolor=blue, pdfstartview=FitH]{hyperref}
\usepackage{esint}

\newtheorem{theorem}{Theorem}
\newtheorem{lemma}{Lemma}
\newtheorem{proposition}{Proposition}

\newtheorem{corollary}{Corollary}

\newtheorem{remark}{Remark}

\title{Effect of the average scalar curvature on Riemannian manifolds}
\begin{document}
\author{Kwok-Kun Kwong}
\address{
School of Mathematics and Applied Statistics\\
University of Wollongong\\
Northfields Ave\\
NSW 2522, Australia}
\email{kwongk@uow.edu.au}
\maketitle

\begin{abstract}
We investigate the effect of the average scalar curvature on the conjugate radius, average area of the geodesic spheres, average volume of the metric balls and the total volume of a closed Riemannian manifold $N$ (or more generally $N$ with finite volume whose negative Ricci curvature integral on $SN$ is finite). For example, we prove that if the average scalar curvature is larger than the lower bound of the normalized Ricci curvature, then we can improve the Bishop-Gromov estimate on the average volume of the metric balls of any size. We also prove the monotone decreasing property of a certain geometric integral when the average scalar curvature has a lower bound. This leads to a comparison theorem of the average total mean curvature of geodesic spheres of radius up to $\mathrm{inj}(N)$.
\end{abstract}
\section{Introduction}
The research in this paper is partly motivated by the following result by Green \cite{average} (obtained independently by Berger), which requires only a lower bound on the average of the scalar curvature $R$ on a closed manifold. For a proof in English, see \cite[Theorem 1.16]{lee2019geometric}.

\begin{theorem}[Green, Berger]\label{green}
Let $\left(N^{n}, g\right)$ be a closed Riemannian manifold whose average scalar curvature $R$ is at least $n(n-1)$. Then the conjugate radius $\mathrm{conj}(N)$ of $(N, g)$ is less than or equal to $\pi$. If it is equal to $\pi$, then $(N, g)$ has constant sectional curvature $1$.
\end{theorem}

Recall that the conjugate radius $\mathrm{conj}(N)$ is the supremum of all $r$ such that for any unit speed geodesic, any two conjugate points along it are at least $r$ units apart. Equivalently, it is the supremum of all $r$ such that the exponential map at every $p \in N$ has nonsingular derivative at every point of the ball $B_0(r) \subset T_{p}N$. In this paper, any Riemannian manifold $(N, g)$ is assumed to be complete, smooth and oriented, and $g$ is smooth as well.

The proof of Theorem \ref{green} is similar to that of the Bonnet-Myers theorem, by making use of the second variation formula for arclength. However, the Bonnet-Myers theorem assumes a lower bound on the Ricci curvature instead of the scalar curvature. In order to obtain some information from the average of the scalar curvature, the new ingredient is to integrate a certain geometric inequality (obtained from the second variation formula) along all $p$ in $N$ and all ``directions'' $\theta\in S_pN$, i.e. along the unit sphere bundle $SN$. The key point is that by Liouville's theorem, the geodesic flow preserves the canonical measure of $SN$, and this will ensure that the integral of the Ricci curvature $\int_{(p, \theta)\in SN} \mathrm{Ric}\left(\gamma_{p, \theta}^{\prime}(t)\right) d \mu_{SN}$ is actually independent of $t$, and becomes just the integral of the scalar curvature up to a multiplicative constant. This gives rise to the question of whether we can ``translate'' some results which require a lower bound on the Ricci curvature to just a lower bound on the average of the scalar curvature, or at least an improvement of those results if, in addition, the average of the scalar curvature is known.

The second motivation comes from (i) the Bishop-Gromov volume comparison theorem, and also (ii) the Taylor expansion of the volume of small geodesic balls about $r=0$, which involves the scalar curvature. The well-known Bishop Gromov volume comparison theorem says that if the (normalized) Ricci curvature is bounded below by a constant, then the volume of the geodesic ball is at most that of the volume of the ball with the same radius in the space form whose curvature is that constant. Counterexamples show that this assumption cannot be weakened to a lower bound on the scalar curvature. On the other hand, the Taylor expansion of the volume of small geodesic balls is
(cf. \cite[Theorem 3.1]{gray1974volume})
\begin{equation}\label{small r vol}
|B_p(r)|=\frac{\omega_{n-1}}{n} r^{n}\left(1-\frac{R(p)}{6(n+2)} r^{2}+O\left(r^4\right)\right),
\end{equation}
where $\omega_{n-1} $ is the area of the unit sphere in $\mathbb{R}^{n}$. This shows that a lower bound on the scalar curvature does have an effect on the volume of those balls when the radius is small. It is a difficult problem to estimate the global effect of the scalar curvature on the volume of geodesic balls when the radius is not small anymore, and an even harder problem to estimate its effect on the total ``size'' of the ambient manifold.

In this paper, we first generalize Green's theorem (Theorem \ref{green}) to the case where we only assume $N$ has finite volume and its negative Ricci curvature has finite integral:
\begin{theorem}[Theorem \ref{thm gen green}]\label{thm intro 0}
Suppose $(N^n, g)$ is complete with finite volume.
Assume the following:
\begin{enumerate}
\item
The integral $\int_{SN} \mathrm{Ric}^-(p, \theta) d\mu_{SN}$ is finite, where $\mathrm{Ric}^-_p(\theta)=\max\{-\mathrm{Ric}(p, \theta), 0\}$ denotes the negative part of the Ricci curvature.
\item
The average scalar curvature $\overline R$ is at least $n(n-1)$.
\end{enumerate}
Then $N$ has conjugate radius at most $ \pi $. The conjugate radius equals $ \pi $ if and only if $N$ has constant sectional curvature $1$.
\end{theorem}

We then establish some average area or volume estimates which involves not only the Ricci curvature, but also the average of the scalar curvature. We will see that if the normalized average scalar curvature on a closed manifold is greater than the (pointwise) lower bound of the Ricci curvature, then our result gives an improvement on the estimate of the average volume of geodesic balls and metric balls by using only the Bishop-Gromov volume comparison. Indeed, the result is quantitative and the average scalar curvature appears as an ``correction'' term to the estimate given by the Bishop-Gromov volume comparison theorem.

For example, we can prove the following result as a special case of Theorem \ref{thm2}.
\begin{theorem} [Theorem \ref{thm2}]\label{thm intro 2}
Let $\left(N^{n}, g\right)$ be a closed Riemannian manifold and $r>0$. Let $\overline R$ be the average of the scalar curvature on $N$.
Suppose $0 \le \mathrm{Ric} \le \kappa$ (where $\kappa>0$), then the average volume $\overline V(r)$ of the metric balls of radius $r$ in $(N, g)$ satisfies
$$\overline {V}(r) \le \frac{\omega_{n-1}}{n} r^{n}-\frac{\omega_{n-1}\overline R}{n \kappa} \int_{0}^{r}\left(1-e^{-\frac{\kappa t^{2}}{6}}\right)t^{n-1} dt.
$$
The equality holds if and only if $N$ is flat and $r\le \mathrm{inj}(N)$.
\end{theorem}
As opposed to \eqref{small r vol}, this inequality is true even beyond the injectivity radius, not only for small $r$. The corresponding comparison against other space forms of curvature $k$ holds as well. Indeed, the $k=1$ case of this result gives an improved upper bound of the total volume if we have a positive lower bound of the Ricci curvature:
\begin{theorem}[Theorem \ref{thm total vol}]\label{thm intro 3}
Let $\kappa>0$ and $\left(N^{n}, g\right)$ be a closed Riemannian manifold with $0 \le \mathrm{Ric}-(n-1)g \le \kappa$, then
$$
|N| \le \left|\mathbb{S}^{n}\right|-\frac{\omega_{n-1}\overline {R_{1}}}{n \kappa} \int_{0}^{\pi}\left(1-e^{-\kappa \sigma_{1}(t)}\right)\sin^{n-1}(t) d t.
$$
Here $\overline {R_1}=\overline R-n(n-1)$ and $\sigma_1(t)$ is the positive function $\frac{1}{2 }(1- t \cot (t))$ on $(0, \pi)$.

The equality holds if and only if $N$ is isometric to $\mathbb S^{n}$.
\end{theorem}

We also have a result for the average over the unit sphere bundle $SN$ of the mean curvature of geodesic spheres in $N$. Let $H(p, r, \theta)$ be the mean curvature at the point $\exp_{p}(r \theta)$ of the geodesic sphere of radius $r$ centered at $p$.
We can prove the following result:
\begin{theorem}[Theorem \ref{thm H}]\label{thm intro 4}
Let $\left(N^{n}, g\right)$ be a closed Riemannian manifold with average scalar curvature $\overline R$. Then the average over $S N$ of the mean curvature of the geodesic spheres of radius $r$ in $N$ satisfies
$$
\fint_{S N} H(p, r, \theta) d \mu_{S N}(p, \theta) \le \frac{n-1}{r}-\frac{\overline R}{3n}r
$$
for any $r>0$ which is smaller than the injectivity radius of $N$. If the equality holds for some $r$, then $(N, g)$ is flat.
\end{theorem}
Interestingly, this result is equivalent to the monotone property of a certain geometric integral. More precisely, it is not hard to see that this is equivalent to
$$
\frac{d}{d r}\left(\fint_{S N} \log \left(\frac{F(p, r, \theta)}{r^{n-1}}\right) d \mu_{S N}(p, \theta)\right) \le -\frac{\overline{R}}{3n}r,
$$
where $F(p, r, \theta)$ is the Jacobian in the geodesic polar coordinates centered at $p$.
Therefore $\fint_{S N} \log \left(\frac{F(p, r, \theta)}{r^{n-1}}\right) d \mu_{S N}(p, \theta)$ is monotone decreasing in $r$ if the average scalar curvature is non-negative. Whether the integral $\fint_{S N} \log \left(\frac{F(p, r, \theta)}{r^{n-1}}\right) d \mu_{S N}(p, \theta)$ (besides $\fint_{SN}\int_{0}^{r}\left(H(p, t, \theta)-\frac{n-1}{t}\right)dt d\mu_{S N}(p, \theta)$) has a simple geometric interpretation is unknown to us.

Note that unlike Theorem \ref{thm intro 2} and Theorem \ref{thm intro 3}, we do not impose any assumption on the curvature of the ambient space in Theorem \ref{thm intro 4}.
This result comes as a surprise to us due to the following reason. For small geodesic spheres, Gray and Vanhecke \cite[Theorem 12.3]{gray1979Riemannian} computed the expansion
\begin{align*}
&\fint_{\mathbb S^{n-1}} H(p, r, \theta)d\theta\\
=&\frac{n-1}{r}-\frac{R(p)}{3 n} r-\frac{1}{90 n(n+2)}\left(3\|\mathrm{Rm}|_p\|^{2}+2\|\mathrm{Ric}|_p\|^{2}+18 \Delta R(p)\right) r^{3}+O(r^5),
\end{align*}
where the $O(r^5)$ term is an explicit but complicated term (which spans three lines!), whose sign is not obvious at all. So while we should expect that for small $r$, the average of the total mean curvature should be not more than $\frac{n-1}{r}-\frac{\overline {R}}{3 n} r$, at least up to order $4$ terms, we shouldn't expect that by ``integrating'' this Taylor expansion, we can obtain an inequality for spheres of all size as the higher order terms may dominate.
The point is that given only a lower bound of the average of the scalar curvature this result gives an estimate of the mean curvature of geodesic spheres of any radius (up to the injectivity radius $\mathrm{inj}(N)$), but in the average sense.

There are several ingredients in the proofs of the main results. Central to the proofs of these results is the Liouville theorem: the canonical measure on the unit sphere bundle $SN$ is preserved under the geodesic flow. This will ensure that the integral of the Ricci curvature $\int_{(p, \theta) \in S N} \mathrm{Ric}\left(\gamma_{p, \theta}^{\prime}(t)\right) d \mu_{S N}$ is actually independent of $t$, and becomes just the average of the scalar curvature up to a multiplicative constant. To prove the area or volume comparison theorem, we also need a Laplacian or Jacobian comparison result which does not require any assumption on the curvature, see Proposition \ref{prop1}. This can be regarded as the infinitesimal version of the volume comparison result. In proving Theorem \ref{thm intro 2} and Theorem \ref{thm intro 3}, another ingredient is a ``reverse'' Jensen inequality (Lemma \ref{lem1}), which enables us to swap the integral sign with the exponential function when integrating the infinitesimal version of the inequalities.

The organization of this paper is as follows. In the next section, we are going to study the effect of the average scalar curvature on various geometric quantities on $N$. We first set up the notation in Subsection \ref{notation}. Then in Subsection \ref{conj} we illustrate the effect on the conjugate radius, which is generalization of the Green's theorem. In Subsection \ref{jac est}, we give the main Jacobian estimate that does not require any assumption on the curvature, which is essentially already contained in \cite{kwong2019quantitative}. In Subsection \ref{area}, we are going to study the effect of the average scalar curvature on the average area of geodesic sphere, average volume of the metric balls and the total volume. Finally we reinterpret Theorem \ref{thm monotone} as an effect of the scalar curvature on the average of the mean curvature of geodesic spheres in Subsection \ref{mean curvature}.\\

\noindent Acknowledgement: We would like to thank Man-Chun Lee for useful discussions and Luen-Fai Tam for his interest. The research of the author is partially supported by the CERL fellowship at University of Wollongong.

\section{Average scalar curvature and its effect on the geometry of $N$}
\subsection{Notation}\label{notation}

Let $(N, g)$ be an $ n$-dimensional Riemannian manifold. Let $ k\in \mathbb R$ and define the functions $s_k$ by
\begin{align*}\label{eq: sk}
s_k(t)=
\begin{cases}
\frac{1}{\sqrt{k}} \sin \left(\sqrt{k}t\right)\quad &\textrm{ if }k>0\\
t\quad &\textrm{ if }k=0\\
\frac{1}{\sqrt{-k}}\sinh \left(\sqrt{-k}t\right)\quad &\textrm{ if }k<0.
\end{cases}
\end{align*}
For $p\in N$, let $\gamma_{p, \theta}(t)$ be the geodesic starting from $ p$ with initial vector $\theta\in S_pN =\{\theta\in T_pN : |\theta|=1\}$. We define
\begin{equation*}
{\mathrm{Ric}}_k:= \mathrm{Ric}-(n-1)k g.
\end{equation*}
For simplicity we will write $ {\mathrm{Ric}}_k(v)$ instead of $ {\mathrm{Ric}}_k(v, v)$. We can regard $\mathrm{Ric}_k$ as a function on the unit sphere bundle $SN$.
Let $R$ be the scalar curvature of $(N, g)$. We also define $ {R}_k =\mathrm{tr}_g({\mathrm{Ric}_k})=R-n(n-1)k$ and $\overline { {R}_k }=\fint_N {R_k}$.

Let $p\in N$ and let $\{t, \theta^1, \cdots, \theta^{n-1}\}$ be the geodesic polar coordinates around the point $\exp_p(t\theta)$, where $\theta \in S_{p} N \cong \mathbb{S}^{n-1}$. We define the Jacobian $F(p, t, \theta):=\sqrt{\mathrm{det}_{n-1}\left(g_{i j}\right)}$, where $g_{i j}:=g\left(\left. d \exp_{p}\right|_{t \theta}\left(\frac{\partial}{\partial \theta^{i}}\right), \left. d \exp_{p}\right|_{t \theta}\left(\frac{\partial}{\partial \theta^{j}}\right)\right)$.
Within the cut locus of $p$, the volume element on $N$ can be expressed as $d \mu_N=F(p, r, \theta) d r d \theta$, where $d \theta$ is the volume element of $\mathbb{S}^{n-1}$. Let $$F_k(r)=s_{k}(r)^{n-1}, $$
which is the corresponding volume density of the space form $\left(\left[0, r_{0}\right) \times \mathbb{S}^{n-1}, d t^{2}+s_{k}(t)^{2} g_{\mathbb{S}^{n-1}}\right)$ in polar coordinates.

For $r$ less than the injectivity radius of $p\in N$, the geodesic ball $B_p(r)$ (resp. geodesic sphere $S_p(r)$) of radius $r$ centered at $p$ is by definition the image under the exponential map $\exp_p$ of the open ball (resp. sphere) in $T_pN$ of radius $r$ centered at $0$.

We denote the $(n-1)$-dimensional area $|\mathbb{S}^{n-1}|$ by $\omega_{n-1}$. The $(n-1)$-dimensional area of the geodesic sphere of radius $r$ centered at $p$ is denoted by $A(p, r)$.
We also use both $\exp(x)$ and $e^x$ interchangeably to denote the exponential function.
\subsection{Effect on conjugate radius}\label{conj}

We first illustrate the idea of Theorem \ref{green}. Indeed, we are going to prove a generalization of it.
The Bonnet-Myers theorem says that if $N$ has a positive Ricci curvature lower bound, then it is compact and its diameter is bounded. This gives rise to the following question: if the volume is finite and the negative part of the Ricci curvature is not too large in a certain sense, what we can say about the ``size'' of the manifold.
The following result gives a result of this kind, and can be regarded as a generalization of Green's theorem. This result is perhaps known to experts, but in any case it illustrates the role of the scalar curvature in the proofs of the remaining results.
\begin{theorem}\label{thm gen green}
Suppose $(N^n, g)$ is complete.
Assume the following:
\begin{enumerate}
\item
{{$(N, g)$ has finite volume and the integral $\int_{SN} \mathrm{Ric}^-(p, \theta) d\mu_{SN}(p, \theta)$ is finite, }}where $\mathrm{Ric}^-(p, \theta)=\max\{-\mathrm{Ric}(p, \theta), 0\}$ denotes the negative part of the Ricci curvature.
\item
The average of the scalar curvature is at least $n(n-1)k$, where {{$k\in(0, \infty]$}}.
\end{enumerate}
Then either {{$\mathrm{conj}(N)=0$}}, or {{$\fint_N Rd\mu_N$ (and hence $k$) is finite and $0<\mathrm{conj}(N)\le\frac{\pi}{\sqrt{k}}$. }}

The conjugate radius equals $\frac{\pi}{\sqrt{k}}$ (for finite $k$) if and only if $N$ has constant sectional curvature $k$.
\end{theorem}

\begin{proof}
If $\mathrm{conj}(N)=0$, then we have nothing to prove. Otherwise, let $l$ be any positive real number not bigger than $\mathrm{conj}(N)$.

Let $p \in N$ and $\theta\in S_{p} N$, and let $\gamma=\gamma_{p, \theta}:[0, l] \to N$ be the unique geodesic with $\gamma(0)=p$ and $\gamma^{\prime}(0)=\theta$. As $\gamma$ has no conjugate points before $l$, for any vector field $X$ along $\gamma$ vanishing at the endpoints, we have $I(X, X) \ge 0$ by the index lemma, where
$$I(X, X)=\int_{0}^{l}\left[\left|X^{\prime}(t)\right|^{2}-\mathrm{Rm}\left(\gamma^{\prime}(t), X(t), X(t), \gamma^{\prime}(t)\right)\right] d t$$
is the index form. Choose $V_{1}, \cdots, V_{n-1}$ so that $\gamma^{\prime}, V_{1}, \cdots, V_{n-1}$ forms a parallel orthonormal basis along $\gamma$ and set $X_{i}=\sin \left(\frac{\pi t}{l}\right) V_{i}$, then we have $I\left(X_{i}, X_{i}\right) \ge 0$ for each $i$ from 1 to $n-1$. Summing this inequality over $i$ gives
$$
(n-1) \frac{\pi^{2}}{2 l}-\int_{0}^{l} \mathrm{Ric}\left(\gamma_{p, \theta}^{\prime}(t)\right) \sin^{2}\left(\frac{\pi t}{l}\right) d t \ge 0.
$$
Now we integrate this inequality over $(p, \theta)\in SN$, then by {{Fubini-Tonelli theorem}},
\begin{equation}\label{fubini}
(n-1) \frac{\pi^{2}}{2 l} \omega_{n-1}|N|-\int_{0}^{l}\left(\int_{SN} \mathrm{Ric}\left(\gamma_{p, \theta}^{\prime}(t)\right) d \mu_{SN}(p, \theta)\right) \sin^{2}\left(\frac{\pi t}{l}\right) d t \ge 0.
\end{equation}
{{We make some remarks on \eqref{fubini}. Recall that the Fubini-Tonelli theorem states that if either the negative (or positive) part $f^-$ of the integrand $f$ satisfies $\int_0^l\int_{SN}f^- d\mu_{SN}dt<\infty$, then $\int_{SN}\int_{0}^{l}f dt d\mu_{SN}=\int_{0}^{l}\int_{SN}f d\mu_{SN}dt$. The negative part of the integrand in \eqref{fubini} is bounded by $ \mathrm{Ric}^-(\gamma_{p, \theta}'(t)) $. We are going to see that $\int_{SN} \mathrm{Ric}^-(\gamma_{p, \theta}'(t))d\mu_{SN} $ is finite. }}{{Indeed, by Liouville theorem}} \cite[p.117]{chavel1984eigenvalues}, $\int_{SN} \mathrm{Ric}^-\left(\gamma_{p, \theta}^{\prime}(t) \right) d \mu_{SN}$ is actually independent of $t$ since the geodesic flow is a diffeomorphism of $SN$ preserving $d \mu_{SN}$, and so it is equal to $\int_{SN} \mathrm{Ric}^-\left(\gamma_{p, \theta}^{\prime}(0) \right) d \mu_{SN}=\int_{S N} \mathrm{Ric}^-(p, \theta) d\mu_{SN}<\infty$. So we can interchange the integral sign and \eqref{fubini} is justified. Therefore by the {{coarea formula}} and the {{Liouville theorem again}},
\begin{equation*}\label{ineq conj}
\begin{split}
(n-1) \frac{\pi^{2}}{2 l} \omega_{n-1}|N|
& \ge \left(\int_{SN} \mathrm{Ric}(p, \theta) d \mu_{SN}\right)\left(\int_{0}^{l} \sin^{2}\left(\frac{\pi t}{l}\right) d t\right) \\
& = \left(\int_{N}\int_{S_pN} \mathrm{Ric}(p, \theta) d\theta d \mu_{N}\right)\frac{l}{2} \\
&=\left(\int_{N} \frac{\omega_{n-1}}{n} R \, d \mu_{N}\right) \frac{l}{2} \\
& \ge (n-1) \frac{lk}{2} \omega_{n-1}|N|.
\end{split}
\end{equation*}
Therefore $k$ is finite and $l \le \frac{\pi}{\sqrt{k}}$. Hence $\mathrm{conj}(N)\le \frac{\pi}{\sqrt{k}}$. The above also shows that $\fint_N Rd\mu_N<\infty$. {{Let us also remark that $R^-(p)\le \fint_{S_pN} \mathrm{Ric}^-(p, \theta)d\theta$ and so $\int_{N}R^-d\mu_N\le \int_N \fint_{S_pN} \mathrm{Ric}^-(p, \theta)d\theta d\mu_N<\infty$. Therefore the Lebesgue integral $\int_N R d\mu_N\in(-\infty, \infty]$ exists. }}

If $\mathrm{conj}(N)=\frac{\pi}{\sqrt{k}}$ and $k\in(0, \infty)$, then every inequality becomes an equality. In particular the vector fields $X$ above become Jacobi fields, and from this one can see that the sectional curvatures $K\left(\gamma^{\prime}, V\right)$ along each geodesic $\gamma$ are all equal to $k$. Therefore the $(N, g)$ has constant curvature $k$.
\end{proof}

This result can also be interpreted  as a lower bound of the total volume given a lower bound of the conjugate radius, slightly generalizing \cite[Theorem 2.1]{tuschmann2019smooth}:
\begin{corollary}
Let $N^n$ be a Riemannian manifold with finite volume and with conjugate radius $\mathrm{conj}(N) \ge l$.
Assume the integral $\int_{S N} \mathrm{Ric}^{-}(p, \theta) d \mu_{S N}(p, \theta)$ is finite.
Then $|N| \ge \frac{l^2}{n(n-1) \pi^2} \int_N R d\mu_N$.
The equality holds if and only if $N$ has constant sectional curvature $\frac{\pi^2}{l^2}$.
\end{corollary}

\subsection{Jacobian and mean curvature estimates}\label{jac est}

\begin{proposition}\label{prop1}
Assume there is no cut point of $p$ along $\gamma_{p, \theta}$ on $[0, r]$. If $s_{k}>0$ on $(0, r]$, then
\begin{equation}\label{prop1 ineq}
F(p, r, \theta) \le \exp \left[-\int_{0}^{r} \int_{0}^{\tau} \frac{s_{k}(t)^{2}}{s_{k}(\tau)^{2}} {\mathrm{Ric}}_{k}\left(\gamma_{p, \theta}^{\prime}(t)\right) d t d \tau\right] F_k(r).
\end{equation}
If the equality holds for all $\theta\in S_pN$, then the geodesic ball of radius $r$ centered at $p$ is isometric to the geodesic ball of radius $r$ in the standard space form of curvature $k$.
\end{proposition}
\begin{proof}
This was proved in \cite[Theorem 1 (3)]{kwong2019quantitative}. Since the equality case was omitted in \cite{kwong2019quantitative}, for completeness we sketch the proof here.

Let $e_{1}, e_{2}, \cdots, e_{n}=\theta$ be a positively oriented orthonormal basis of $T_{p} N$ and $E_{i}$ be the parallel translation of $e_{i}$ along $\gamma_{\theta}$. Define $\left\{Y_{i}^{r, \theta}(t)\right\}_{i=1}^{n-1}$ to be the unique Jacobi fields along $\gamma_{\theta}$ with $Y_{i}^{r, \theta}(0)=0$ and $Y_{i}^{r, \theta}(r)=E_{i}(r)$. For convenience we simply denote $Y_{i}^{r, \theta}$ by $Y_{i}$ and $\gamma_{p, \theta}$ as $\gamma$.

It is not hard to see that the $(n-1)$-dimensional Jacobian satisfies $F(p, t, \theta)=\frac{\mathrm{det}\left(Y_{1}(t), \cdots, Y_{n-1}(t)\right)}{\mathrm{det}\left(Y_{1}^{\prime}(0), \cdots, Y_{n-1}^{\prime}(0)\right)}$. We have the formula (cf. \cite[p. 460]{heintze1978general})
\begin{equation}\begin{split}\label{(2.4)}
\frac{\partial }{\partial r}\left(\log F (p, r, \theta)\right)
=&\frac{\partial }{\partial r}\left[\log \left(\mathrm{det}\left(Y_{1}, \cdots, Y_{n-1}\right)\right)\right]\\
=&\sum_{i=1}^{n-1} \int_{0}^{r}\left(\left\langle Y_{i}^{\prime}, Y_{i}^{\prime}\right\rangle-\left\langle \textrm{Rm}\left(Y_{i}, \gamma^{\prime}\right) \gamma^{\prime}, Y_{i}\right\rangle\right) d t \\
=&\sum_{i=1}^{n-1} I\left(Y_{i}, Y_{i}\right)
\end{split}\end{equation}
where $I$ is the index form.
Let $X_{i}(t)=\frac{s_{k}(t)}{s_{k}(r)} E_{i}(t)$. Then by the index lemma
\begin{equation}\label{(2.5)}
I\left(Y_{i}, Y_{i}\right) \le I\left(X_{i}, X_{i}\right).
\end{equation}
By integration by parts,
\begin{equation}\begin{split}\label{(2.6)}
I\left(X_{i}, X_{i}\right) &=\int_{0}^{r}\left(-\left\langle X_{i}^{\prime \prime}, X_{i}\right\rangle-\left\langle \mathrm{Rm}\left(X_{i}, \gamma^{\prime}\right) \gamma^{\prime}, X_{i}\right\rangle\right) d t+\left\langle X_{i}(r), X_{i}^{\prime}(r)\right\rangle \\
&=-\int_{0}^{r} \frac{s_{k}(t)^{2}}{s_{k}(r)^{2}} (\langle \mathrm{Rm} (E_i, \gamma')\gamma', E_i\rangle -k)dt+\frac{s_{k}^{\prime}(r)}{s_{k}(r)}.
\end{split}\end{equation}
Summing \eqref{(2.6)} over $i=1, \cdots, n-1$ and combining with \eqref{(2.4)}, \eqref{(2.5)}, we have
\begin{equation}\label{F'}
\frac{\partial}{\partial r}(\log F(p, r, \theta)) \le -\int_{0}^{r} \frac{s_{k}(t)^2}{s_{k}(r)^2} {\mathrm{Ric}}_{k}\left(\gamma_{p, \theta}^{\prime}(t)\right) d t+\left(\log F_k\right)^{\prime}(r).
\end{equation}

Note that $\log F(p, r, \theta)-\log F_k(r) \rightarrow 0$ as $r\to 0^+$, so integrating this inequality gives the result.

Suppose the equality in \eqref{prop1 ineq} holds for all $\theta\in S_pN$. By the equality case of the index lemma, $X_i$ are Jacobi fields, and so $k \frac{s_k(t)}{s_k(r)}E_i=-\nabla_{\gamma'}\nabla_{\gamma'} X(t)=\mathrm{Rm}(X_i, \gamma')\gamma'=\frac{s_k(t)}{s_k(r)}\mathrm{Rm}(E_i, \gamma')\gamma' $. This implies that the sectional curvature of all planes spanned by $\gamma'$ and $E_i$ are equal to $k$. Therefore $ {\mathrm{Ric}}_{k}\left(\gamma_{p, \theta}^{\prime}(t)\right)=0$ and so the equality is actually $F(p, r, \theta)=F_k(r)$.

Moreover, let $(r, \{\theta^i\}_{i=1}^{n-1})$ be a geodesic polar coordinates, then from the above, the Jacobi field $J_i(t)=d\exp_p|_{t\theta}\left(\frac{\partial }{\partial \theta^i}\right)$ along $\gamma$ can be expressed as $s_k(t)V_i(t)$ for a parallel normal vector field $V_i(t)$ along $\gamma$. Moreover, $V_i(0)=J_i'(0)$. Let $\overline g$ be the standard Euclidean metric in $T_pN$. Then
\begin{align*}
g\left(J_i (r), J_j (r)\right)
=s_k(r)^2 g\left(V_i(r), V_j(r)\right)
=&s_k(r)^2 g\left(V_i(0), V_j(0)\right)\\
=&s_k(r)^2 g\left(J_i'(0), J_j'(0)\right)\\
=&s_k(r)^2 \overline g\left(J_i'(0), J_j'(0)\right)\\
=&s_k(r)^2 \overline g\left(\frac{\partial }{\partial \theta^i}, \frac{\partial }{\partial \theta^j}\right).
\end{align*}
So in polar coordinates, the metric can be expressed as $dt^2+s_{k}(t)^{2} g_{\mathbb{S}^{n-1}}$, i.e. the geodesic ball of radius $r$ is isometric to the geodesic ball of radius $r$ in the standard space form of curvature $k$.
\end{proof}
\begin{proposition}\label{prop2}
Assume there is no cut point of $p$ along $\gamma_{p, \theta}$ on $[0, r]$. Let $H(p, r, \theta)$ be the mean curvature at the point $\exp_{p}(r\theta)$ of the geodesic sphere of radius $r$ centered at $p$. If $s_{k}>0$ on $(0, r]$, then
\begin{equation*}\label{prop2 ineq}
H(p, r, \theta) \le \frac{{F_k}^{\prime}(r)}{{F_k}(r)}-\int_{0}^{r} \frac{s_{k}(t)^2}{s_{k}(r)^2} {\mathrm{Ric}}_{k}\left(\gamma_{p, \theta}^{\prime}(t)\right) d t.
\end{equation*}
If the equality holds for all $\theta\in S_pN$, then the geodesic ball of radius $r$ centered at $p$ is isometric to the geodesic ball of radius $r$ in the standard space form of curvature $k$.
\end{proposition}
\begin{proof}
From \eqref{F'}, we have $ \frac{\partial}{\partial r}(\log F(p, r, \theta))
\le \frac{d}{d r}(\log F_k(r))-\int_{0}^{r} \frac{s_{k}(t)^{2}}{s_{k}(r)^{2}} {\mathrm{Ric}_{k}}\left(\gamma_{p, \theta}^{\prime}(t)\right) d t$. It is well-known that $H(p, r, \theta)={\frac{\partial}{\partial r}(\log F(p, r, \theta))}$, from which the result follows. The equality case can be proved in the same way as Proposition \ref{prop1}.
\end{proof}
\subsection{Effect on area and volume}\label{area}

The area of a geodesic sphere in the space form of curvature $k$ is given $A_k(r):=\omega_{n-1}F_k(r)$, and the volume of a geodesic ball in the space form of curvature $k$ is given $V_k(r):=\int_{0}^{r}A_k(t)dt$.
Let $\overline A(r):=\fint_N A(\cdot, r)d\mu$ which is the average area of the geodesic spheres of radius $r$ in $N$. We are going to compare $\overline A(r)$ with $A_k(r)$. When $r$ is less than the injectivity radius of $p$, then $\int_{0}^{r} A(p, t)dt$ is the volume of the geodesic ball $B_p(r)$ centered at $p$ with radius $r$. Later on we will also estimate the total volume of $N$, and for this purpose we will consider not just the average of the volume of the geodesic balls, but also the metric balls. By a metric ball, we mean $\mathcal B_p(r):=\{x\in N: d(x, p)\le r\}$. We define $V(p, r):=|\mathcal B_p(r)|$ and $\overline V(r):=\fint_{N} V(\cdot, r)d\mu$ be its average. Of course, for small $r$, it is also the average of the volume of the geodesic balls.

First of all, we obtain a geometric quantity which is monotone decreasing when the average of $R_k$ is non-negative.
\begin{theorem}\label{thm monotone}
Let $\left(N^{n}, g\right)$ be a closed Riemannian manifold. Suppose $r>0$ is smaller than the injectivity radius of $N$, with $r<\frac{\pi}{\sqrt{k}}$ if $k>0$.
Then
\begin{equation}\label{ineq monotone}
\frac{d}{d r}\left(\fint_{S N} \log \left(\frac{F(p, r, \theta)}{F_{k}(r)}\right) d \mu(p, \theta)\right) \le -\frac{\overline{R_{k}}}{n} \phi_k(r),
\end{equation}
where $\phi_k$ is a non-negative function given by
\begin{equation}\label{phi}
\phi_{k}(r)=
\begin{cases}\frac{1}{2}\left(r \csc^{2}(\sqrt{k} r)-\frac{\cot (\sqrt{k} r)}{\sqrt{k}}\right) & \text { if } k>0 \\
\frac{r}{3} & \text { if } k=0 \\
\frac{1}{2}\left(\frac{ \coth (\sqrt{-k} r)}{\sqrt{-k}}-r\, \mathrm{csch}^{2}(\sqrt{-k} r)\right) & \text { if } k<0.
\end{cases}
\end{equation}
In particular, if $ \overline {R_k}\ge 0$, then the quantity $\fint_{S N} \log \left(\frac{F(p, r, \theta)}{F_{k}(r)}\right) d \mu(p, \theta)$ is monotone decreasing in $r$.

The equality holds for some $r>0$ if and only if $(N, g)$ has constant sectional curvature $k$.
\end{theorem}

\begin{proof}
First of all, we note that if $k>0$ and $\overline {R_k}\ge 0$, then by Theorem \ref{green}, $r<\frac{\pi}{\sqrt{k}}$.
From Proposition \ref{prop2},
\begin{align*}
\frac{\partial }{\partial r}\left(\log \left(\frac{F(p, r, \theta)}{F_k(r)}\right)\right)
=&\frac{\partial }{\partial r}\left(\log {F(p, r, \theta)}-\log {F_k(r)} \right)\\
\le&-\int_{0}^{r} \frac{s_{k}(t)^{2}}{s_{k}(r)^{2}} {\mathrm{Ric}}_{k}\left(\gamma_{p, \theta}^{\prime}(t)\right) d t.
\end{align*}
Integrating over $(p, \theta)\in SN$ then gives
\begin{align*}
\frac{d }{d r} \left({\fint_{SN}\log\left(\frac{ F(p, r, \theta)}{ F_{k}(r)}\right)d \mu(p, \theta)} \right)
\le&- \fint_{SN} \int_{0}^{r} \frac{s_{k}(t)^{2}}{s_{k}(r)^{2}} {\mathrm{Ric}}_{k}\left(\gamma_{p, \theta}^{\prime}(t)\right) d t d\mu_{SN}\\
=&- \int_{0}^{r} \frac{s_{k}(t)^{2}}{s_{k}(r)^{2}}\fint_{SN} {\mathrm{Ric}}_{k}\left(\gamma_{p, \theta}^{\prime}(t)\right)d\mu_{SN} d t.
\end{align*}
Observe that the integral $\int_{SN} {\mathrm{Ric}}_{k}\left(\gamma_{p, \theta}^{\prime}(t)\right) d \mu$ is actually independent of $t$ since the geodesic flow is a diffeomorphism of $SN$ preserving $d \mu_{SN}$ \cite[p.117]{chavel1984eigenvalues}.
So we have $\fint_{SN} {\mathrm{Ric}}_{k}\left(\gamma_{p, \theta}^{\prime}(t)\right) d \mu_{SN}=\fint_{SN} {\mathrm{Ric}}_{k}\left(\gamma_{p, \theta}^{\prime}(0)\right) d \mu_{SN}=\frac{1}{n|N|}\int_N {R}_k d \mu_N=\frac{1}{n}\overline { {R}_k }$.

\begin{equation*}\label{monotone}
\begin{split}
\frac{d}{d r}\left(\fint_{S N} \log \left(\frac{F(p, r, \theta)}{F_{k}(r)}\right) d \mu(p, \theta)\right)
\le&-\frac{ \overline{R_{k}}}{n}\int_{0}^{r} \frac{s_{k}(t)^{2}}{s_{k}(r)^{2}}dt.
\end{split}
\end{equation*}

By direct computation,
\begin{equation*}
\phi_{k}(r):=
\int_{0}^{r} \frac{s_{k}(t)^{2}}{s_{k}(r)^{2}} d t
= \begin{cases}
\frac{1}{2}\left(r \csc^{2}(\sqrt{k} r)-\frac{\cot (\sqrt{k} r)}{\sqrt{k}}\right) & \text { if } k>0 \\
\frac{r}{3} & \text { if } k=0 \\
\frac{1}{2}\left(\frac{ \coth (\sqrt{-k} r)}{\sqrt{-k}}-r \mathrm{csch}^{2}(\sqrt{-k} r)\right) & \text { if } k<0
\end{cases}.
\end{equation*}
If the equality holds for some $r>0$, then by Proposition \ref{prop2} every geodesic ball of radius $r$ is isometric to the geodesic ball of radius $r$ in the standard space form of curvature $k$, and so $(N, g)$ has constant sectional curvature $k$.
\end{proof}
\begin{corollary}
Let $\left(N^{n}, g\right)$ be a closed Riemannian manifold. Suppose $r>0$ is smaller than the injectivity radius of $N$, with $r<\frac{\pi}{\sqrt{k}}$ if $k>0$. Then
$$
\min_{(p, \theta)} F(p, r, \theta) \le \exp \left(-\frac{\overline{R_{k}}}{n} \psi_{k}(r)\right) F_{k}(r),
$$
where $\psi_k$ is a non-negative function given by
$\psi_{k}(r)=
\begin{cases}
\frac{1}{2}\left(\frac{1}{k}-\frac{r \cot (\sqrt{k} r)}{\sqrt{k}}\right)\quad &\textrm{if }k>0\\
\frac{r^2}{6}\quad &\textrm{if }k=0\\
\frac{1}{2}\left(\frac{1}{k}+\frac{r\coth (\sqrt{-k} r)}{\sqrt{-k}}\right)\quad &\textrm{if }k<0.
\end{cases}
$

The equality holds if and only if $(N, g)$ has constant sectional curvature $k$.
\end{corollary}
\begin{proof}
Integrating \eqref{ineq monotone}, we have
\begin{align*}
-\frac{\overline{R_{k}}}{n} \psi_{k}(r)
\ge \fint_{S N} \log \left(\frac{F(p, r, \theta)}{F_{k}(r)}\right) d \mu(p, \theta)
\ge \log \left(\frac{\min_{(p, \theta)} F(p, r, \theta)}{ F_{k}(r)}\right),
\end{align*}
where $\psi_k(r)=\int_{0}^{r}\phi_k(t)dt
$ is given by the result. From this the result follows.
The equality case follows from the equality case of Theorem \ref{thm monotone}.
\end{proof}

The following inequality is inspired by \cite{takahashi1999inverse}, and can be regarded as an ``inverse'' Jensen inequality.
\begin{lemma}\label{lem1}
Let $f$ be a measurable function on a probability measure space $(\Omega, \mathcal {F}, \mu)$ with $m_1\le f\le m_2$. Then
\begin{equation}\label{3}
\int_\Omega e^f d \mu \le a \int_\Omega f d \mu +b,
\end{equation}
where
\begin{equation*} \label{ab}
\begin{cases}
a=\frac{e^{m_2}-e^{m_1}}{m_2-m_1}, \quad &b=\frac{m_2 e^{m_1}-m_1 e^{m_2}}{m_2-m_1} \quad \textrm{if }m_1<m_2, \\
a=e^m, \quad &b=(1-m)e^m \quad \textrm{if }m_1=m_2=m.
\end{cases}
\end{equation*}
The equality holds if and only if $\{f=m_1\}$ or $\{f=m_2\}$ has measure $1$.
\end{lemma}

\begin{proof}
The constants $a$ and $b$ are chosen such that
$e^{m_2}=a m_2+b$ and $e^{m_1}=a m_1+b$, and so $e^{y} \le a y+b$ for all $y\in[m_1, m_2]$ by convexity.
Therefore
\begin{equation*}
\int_{\Omega} e^f d \mu \le \int_{\Omega}(a f+b) d \mu=a \int_{\Omega} f d \mu+b.
\end{equation*}
\end{proof}

The following theorem is one of the main results in this paper. The condition \eqref{cond1}, which is a technical condition to ensure the commutativity of the integrals, can be replaced by the condition that $(N, g)$ is closed if desired.
\begin{theorem}\label{thm1}
Let $(N^n, g)$ be a complete Riemannian manifold.
Assume the following:
\begin{enumerate}
\item \label{cond0}
$r>0$ is a number smaller than the injectivity radius of $N$, with $r<\frac{\pi}{\sqrt{k}}$ if $k>0$.
\item \label{cond1}
{{$(N, g)$ has finite volume and the integral $\int_{S N} \mathrm{Ric}_k^{-}(p, \theta) d \mu(p, \theta)_{SN}$ is finite, where $\mathrm{Ric}_k^-(p, \theta)=\max\{-\mathrm{Ric}_k(p, \theta), 0\}$. }}
\item\label{cond2}
We have $c_1\le \int_{0}^{r} \int_{0}^{\tau} \frac{s_{k}(t)^{2}}{s_{k}(\tau)^{2}} {\mathrm{Ric}}_{k}\left(\gamma_{p, \theta}^{\prime}(t)\right) d t d \tau\le c_2$ for all $(p, \theta)\in SN$.
\end{enumerate}
Then the average area of the geodesic spheres of radius $r$ satisfies
\begin{equation}\label{3'}
\overline A(r) \le \left(b-a \frac{\overline{ {R}_k}}{n} \sigma_{k}(r)\right) A_k(r).
\end{equation}

Here $ \sigma_{k}(r)=
\begin{cases}\frac{1}{2 k}(1-\sqrt{k} r \cot (\sqrt{k} r)) & \text { if } k>0 \\ \frac{r^{2}}{6} & \text { if } k=0 \\ \frac{1}{2 k}(1-\sqrt{-k} r \mathrm{coth}(\sqrt{-k} r)) & \text { if } k<0
\end{cases}
$, and $a=\frac{e^{-c_1}-e^{-c_2}}{c_2-c_1}$, $b=\frac{c_2 e^{-c_1}-c_1e^{-c_2}}{c_2-c_1}$ if $c_1<c_2$.
If $c_1=c_2=c$, we set $a=e^{-c}$ and $b=(1+c) e^{-c}$.

The equality holds if and only if $(N, g)$ has constant sectional curvature $k$.
\end{theorem}

\begin{proof}
From Proposition \ref{prop1},
\begin{equation*}\label{ineq F}
F(p, r, \theta) \le \exp \left[-\int_{0}^{r} \int_{0}^{\tau} \frac{s_{k}(t)^2}{s_{k}(\tau)^2} {\mathrm{Ric}}_{k}\left(\gamma_{p, \theta}^{\prime}(t)\right) d t d \tau\right] F_k(r).
\end{equation*}
Integrating over $(p, \theta) \in SN$ against the probability measure $ \frac{d \mu_{SN}}{|N| \omega_{n-1}} $, we have
\begin{equation}\label{ineq}
\begin{split}
&\frac{1}{|N| \omega_{n-1}} \int_{(p, \theta) \in SN} F(p, r, \theta) d \mu_{SN}\\
\le& \left(\fint_{SN} \exp \left[-\int_{0}^{r} \int_{0}^{\tau} \frac{s_{k}(t)^{2}}{s_{k}(\tau)^{2}} {\mathrm{Ric}}_{k}\left(\gamma_{p, \theta}^{\prime}(t)\right) d t d \tau\right]d \mu_{SN}\right) F_k(r) \\
\le& \left(-a\fint_{SN}\int_{0}^{r} \int_{0}^{\tau} \frac{s_{k}(t)^{2}}{s_{k}(\tau)^{2}} {\mathrm{Ric}}_{k}\left(\gamma_{p, \theta}^{\prime}(t)\right) d t d \tau d \mu_{SN} +b\right)F_k(r)\\
=&\left(-a\int_{0}^{r}\int_{0}^{\tau}\frac{s_{k}(t)^{2}}{s_{k}(\tau)^{2}} \fint_{SN} {\mathrm{Ric}}_{k}\left(\gamma_{p, \theta}^{\prime}(t)\right)d \mu_{SN}\, d t d \tau +b\right)F_k(r).
\end{split}
\end{equation}
Here we have used \eqref{3} in the second line.
{{Let us comment on the interchange of the integrals in the last line. We will see that on $D=\{0<\tau<r, \, 0<t<\tau\}$, $\frac{s_k(t)^2}{s_k\tau)^2}$ is uniformly bounded by a constant $C>0$. Obviously we only need to consider the local boundedness of this function near $(0, 0)$. There exists $\delta>0$ and $\varepsilon\in(0, 1)$ such that for $0<x<\delta$, we have $\varepsilon x\le s_k(x)\le \frac{1}{\varepsilon}x $. As $D=\{(m\tau, \tau): 0<\tau<r, 0<m<1\}$, we have
$0\le\frac{s_k(m\tau)}{s_k(\tau)}\le \frac{1}{\varepsilon^2}m\le \frac{1}{\varepsilon^2}$ if $0<\tau<\delta$. From this the claim follows. The negative part of the integrand $\frac{s_{k}(t)^{2}}{s_{k}(\tau)^{2}} \mathrm{Ric}_{k}\left(\gamma_{p, \theta}^{\prime}(t)\right)$ is uniformly bounded by $C\mathrm{Ric}_{k}^-\left(\gamma_{p, \theta}^{\prime}(t)\right)$ on $D=\{0<\tau<r, \, 0<t<\tau\}$ and as in the proof of Theorem \ref{thm gen green}, we can interchange the integral signs. }}

Observe that the integral $\int_{SN} {\mathrm{Ric}}_{k}\left(\gamma_{p, \theta}^{\prime}(t)\right) d \mu_{SN}$ is actually independent of $t$ since the geodesic flow is a diffeomorphism of $SN$ preserving $d \mu_{SN}$.
We then have
\begin{equation}\label{average Ric}
\fint_{SN} {\mathrm{Ric}}_{k}\left(\gamma_{p, \theta}^{\prime}(t)\right) d \mu_{SN}
=\fint_{SN} {\mathrm{Ric}}_{k}(p, \theta) d \mu_{SN}
=\frac{1}{n|N|}\int_N {R}_k d \mu_N=\frac{1}{n}\overline { {R}_k }.
\end{equation}
Note that $\overline {R_k}\in(-\infty, \infty]$ exists by the argument in Theorem \ref{thm gen green}.
Combining \eqref{ineq} and \eqref{average Ric} then gives
\begin{equation*}
\begin{aligned}
\frac{1}{|N| \omega_{n-1}} \int_{(p, \theta) \in SN} F(p, r, \theta) d \mu_{SN}
& \le \left(-a\frac{\overline{ {R}_{k}}}{n} \int_{0}^{r}\int_{0}^{\tau}\frac{s_{k}(t)^{2}}{s_{k}(\tau)^{2}} d t d \tau +b\right)F_k(r).
\end{aligned}
\end{equation*}
We compute
\begin{equation}\begin{split}\label{sigma}
\sigma_{k}(r):=
\int_{0}^{r} \int_{0}^{\tau} \frac{s_{k}(t)^{2}}{s_{k}(\tau)^{2}} d t d \tau
=
\begin{cases}
\frac{1}{2 k}(1-\sqrt{k} r \cot (\sqrt{k} r))&\quad \textrm{ if }k>0\\
\frac{r^2}{6} &\quad \textrm{ if }k=0\\
\frac{1}{2k} \left(1 -\sqrt{-k}r \coth \left(\sqrt{-k} r \right)\right) &\quad \textrm{ if }k<0.
\end{cases}
\end{split}\end{equation}
So
\begin{equation*}\label{comp}
\frac{1}{ \omega_{n-1}} \fint_{N} A(p, r) d \mu(p) \le \left(-a \frac{\overline{ {R}_k}}{ n}\sigma_k(r) +b\right) F_k(r).
\end{equation*}

Suppose the equality holds. Then from the equality case of Proposition \ref{prop1}, we deduce that all geodesic balls $B(p, r)$ are isometric to the geodesic ball of radius $r$ in the space form of curvature $k$. Therefore $(N, g)$ is a space form of curvature $k$.
\end{proof}

\begin{corollary}\label{cor1}
Let $\left(N^{n}, g\right)$ be a {{complete}} Riemannian manifold. Assume that the {{conditions \eqref{cond0} and \eqref{cond1}}} in Theorem \ref{thm1} hold.
\begin{enumerate}
\item \label{item 1}
Suppose $0 \le {\mathrm{Ric}}_{k} \le \kappa$, then $\overline A(r) \le \left(1-\frac{1-e^{-\kappa \sigma_{k}(r)}}{n \kappa} \overline {R_k}\right) A_{k}(r)$ and for all $s>0$,
\begin{align}\label{Chebyshev}
\frac{|\{p\in N:A(p, r)\ge s A_k(r)\}|}{|N|}\le \frac{1}{s}\left(1-\frac{1-e^{-\kappa \sigma_{k}(r)}}{n \kappa} \overline {R_k}\right).
\end{align}
\item \label{item 2}
Suppose $\kappa_1\le {\mathrm{Ric}}_{k}\le \kappa_2$ and $\frac{ \overline {{R}_{k}}}{n} \ge (1-t) \kappa_{1}+t \kappa_{2}$, where $t=\frac{e^{-\kappa_{1} \sigma_{k}(r)}-1}{e^{-\kappa_{1} \sigma_{k}(r)}-e^{-\kappa_{2} \sigma_{k}(r)}}$ and $\sigma_k$ is defined by \eqref{sigma}, then $\overline {A}(r) \le A_{k}(r)$.
\end{enumerate}
The equality holds if and only if $(N, g)$ has constant sectional curvature $k$.
\end{corollary}
\begin{proof}
We just prove \eqref{item 1}, as \eqref{item 2} is similar. In this case we can set $c_{1}=0$ and $c_{2}=\sigma_{k}(r) \kappa$ in Theorem \ref{thm1}. So the constants $a$ and $b$ in \eqref{3'} satisfy $b=1$ and $0 \le a=\frac{1-e^{-\kappa \sigma_{k}(r)}}{\kappa \sigma_{k}(r)} \le 1$. From this we conclude that $\frac{1}{\omega_{n-1}} \fint_{N} A(p, r) d \mu(p) \le \left(1-\frac{1-e^{-\kappa \sigma_{k}(r)}}{n \kappa} \overline{{R}_{k}}\right) F_{k}(r)$.

The inequality \eqref{Chebyshev} is just an application of the Chebyshev's inequality.
\end{proof}

By integrating the inequality in Theorem \ref{thm1}, we can obtain an average volume comparison result:
\begin{theorem} \label{thm2}
Let $\left(N^{n}, g\right)$ be a {{complete}} Riemannian manifold.
Assume that the condition \eqref{cond1} in Theorem \ref{thm1} holds and in addition, that
\begin{enumerate}
\item[(3')]\label{cond2'}
$c_1(r)\le \int_{0}^{r} \int_{0}^{\tau} \frac{s_{k}(t)^{2}}{s_{k}(\tau)^{2}} {\mathrm{Ric}}_{k}\left(\gamma_{p, \theta}^{\prime}(t)\right) d t d \tau\le c_2(r)$ for all $(p, \theta)\in SN$.
\end{enumerate}

Suppose $r>0$ and assume $r\le\frac{\pi}{\sqrt{k}}$ if $k>0$.
Then the average volume of the metric balls of radius $r$ satisfies
$$
\overline V(r) \le \int_{0}^{r}\left(b(t)-a(t) \frac{\overline{R_{k}}}{n} \sigma_{k}(t)\right) A_{k}(t) d t
$$
where $\sigma_k$ are given in Theorem \ref{thm1}.
Here, $a(t)=\begin{cases}
\frac{e^{-c_{1}(t)}-e^{-c_{2}(t)}}{c_{2}(t)-c_{1}(t)}&\quad \textrm{if }c_1(t)<c_2(t)\\
e^{-c(t)}&\quad \textrm{if }c_1(t)=c_2(t)=c(t)
\end{cases}
$ and
$b(t)=
\begin{cases}
\frac{c_{2}(t) e^{-c_{1}(t)}-c_{1}(t) e^{-c_{2}(t)}}{c_{2}(t)-c_{1}(t)} &\quad \textrm{if }c_1(t)<c_2(t)\\
e^{-c(t)} &\quad \textrm{if }c_1(t)=c_2(t)=c(t).
\end{cases}
$

The equality holds if and only if $(N, g)$ has constant sectional curvature $k$ and $r\le \mathrm{inj}(N)$.

In particular, if \eqref{item 1} or \eqref{item 2} in Corollary \ref{cor1} holds, then $\overline V(r)\le V_k(r)$.
\end{theorem}
\begin{proof}
First of all, note that as long as $\gamma_{p, \theta}$ has no conjugate point, then the proof of Proposition \ref{prop1} holds for $r$ up to its first conjugate point.

Now, we consider the volume of the metric ball $\mathcal B_p(r)$ instead of just $B_p(r)$. Let $\widetilde{F}(p, r, \theta):=F(p, r, \theta) \chi_{[0, c(p, \theta))}(r)$. Here $c(p, \theta)=\rho$, where $\gamma_{p, \theta}(\rho)$ is the first conjugate point to $p$ along $\gamma_{p, \theta}$.
We can rewrite the inequality \eqref{prop1 ineq} as
\begin{equation}\label{prop1 ineq'}
\widetilde F(p, r, \theta)\le \exp \left[-\int_{0}^{r} \int_{0}^{\tau} \frac{s_{k}(t)^{2}}{s_{k}(\tau)^{2}} \mathrm{Ric}_{k}\left(\gamma_{p, \theta}^{\prime}(t)\right) d t d \tau\right] F_{k}(r)
\end{equation}
for all $r$ if $k\le 0$ and for $r\le \frac{\pi}{\sqrt k}$ if $k>0$. This is because for $r\ge c(p, \theta)$, the LHS is zero but the RHS is always non-negative.

Obviously, the set $ \left\{(t, \theta):\theta\in S_pN, t\le \min\{r, c(p, \theta)\}\right\}$ in geodesic polar coordinates covers $\mathcal B_p(r)$ by the exponential map $\exp_p$. Therefore
\begin{equation} \label{Bp}
|\mathcal B_p(r)|\le \int_{0}^{r}\int_{S_pN}\widetilde{F}(p, t, \theta) d\theta dt.
\end{equation}
So from \eqref{Bp} and \eqref{prop1 ineq'}, we can proceed as in the proof of Theorem \ref{thm1} to conclude that
\begin{align*}
\overline V(r)=\fint_{N}\left|\mathcal{B}_{p}(r)\right|d\mu(p)
=&\int_{0}^{r}\fint_{(p, \theta) \in S N} \widetilde F(p, t, \theta) d \mu_{S N} dt\\
\le& \int_{0}^{r}\left(b(t)-a(t) \frac{\overline{R_{k}}}{n} \sigma_{k}(t)\right) A_{k}(t)dt.
\end{align*}

The equality case is the same as Theorem \ref{thm1}, {{at least when $r\le \mathrm{inj}(M)$. On the other hand, it is obvious that beyond the injectivity radius, the equality cannot hold as \eqref{prop1 ineq'} is strict. }}
\end{proof}

Assume that $ {\mathrm{Ric}}_k\ge 0$, then by Bishop-Gromov comparison theorem, we have $\overline {A}(r)\le A_k(r)$ and $\overline {V}(r)\le V_k(r)$. We see that Corollary \ref{cor1} and Theorem \ref{thm2} gives an improvement of this estimate.
\begin{remark}
Counterexamples show that even if we are considering the average area $\overline A$ or volume $\overline V$, we cannot drop the assumption on $\mathrm{Ric}_k$ in Corollary \ref{cor1} or Theorem \ref{thm2}. Take $N=\mathbb S^2\times (\mathbb H^2/\Gamma)$ which has scalar curvature $0$, where $\Gamma<PSL(2, \mathbb R)$ is Fuchsian and $\mathbb H^2/\Gamma$ is closed. It is not hard to see that for small enough $r$, $\overline V(r)=4 \pi^{2} \int_{0}^{r}\left(1-\cos \left(\sqrt{r^{2}-t^{2}}\right)\right) \sinh (t) dt$. Either from this, or by \cite[Theorem 3.1]{gray1974volume}, we see that $\overline V(r)=\frac{\omega_3}{4} r^{4}\left(1+\frac{1}{3456} r^{4}+O\left(r^{6}\right)\right)>\frac{\omega_3}{4}r^4$.
\end{remark}

\begin{theorem}\label{thm total vol}
Let $k, \kappa>0$ and $\left(N^{n}, g\right)$ be a closed Riemannian manifold with $0 \le \mathrm{Ric}_{k} \le \kappa$, then
$$
|N| \le \left|\mathbb{S}^{n}\left(\frac{1}{\sqrt{k}}\right)\right|- \frac{\overline{R_{k}}}{n\kappa}\int_{0}^{\frac{\pi}{\sqrt{k}}}\left(1-e^{-\kappa \sigma_{k}(t)}\right) A_{k}(t) d t.
$$
The equality holds if and only if $N$ is isometric to $\mathbb S^{n}\left(\frac{1}{\sqrt{k}}\right)$.
\end{theorem}
\begin{proof}

Notice that $\mathcal B_p(\mathrm{diam}(N))=N$ for any $p\in N$. By Bonnet-Myers theorem, $\mathrm{diam}(N) \le \frac{\pi}{\sqrt{k}}$. In particular, by Theorem \ref{thm2} we have
\begin{align*}
|N|
=|\mathcal B_p(\mathrm{diam}(N))|
\le& \int_{0}^{\mathrm{diam}(N)}\left(1-\frac{1-e^{-\kappa \sigma_k(t)}}{n \kappa} \overline{R_k}\right) A_k(t) d t\\
\le& \int_{0}^{\frac{\pi}{\sqrt{k}}}\left(1-\frac{1-e^{-\kappa \sigma_k(t)}}{n \kappa} \overline{R_k}\right) A_k(t) d t\\
=& V_k\left(\frac{\pi}{\sqrt{k}}\right) -\frac{\overline R_k}{n\kappa}\int_{0}^{\frac{\pi}{\sqrt{k}}}\left({1-e^{-\kappa \sigma_k(t)}} \right) A_k(t) d t\\
=& \left|\mathbb S^n\left(\frac{1}{\sqrt{k}}\right)\right|-\frac{\overline R_k}{n\kappa}\int_{0}^{\frac{\pi}{\sqrt{k}}}\left({1-e^{-\kappa \sigma_k(t)}} \right) A_k(t) d t.
\end{align*}
If the inequality is an equality, then the diameter is $\frac{\pi}{\sqrt{k}}$ and so $N$ is the sphere $\mathbb S^{n}\left(\frac{1}{\sqrt{k}}\right)$ by Cheng's diameter theorem.
\end{proof}

Obviously, Theorem \ref{thm total vol} can be generalized to the case where $k\le 0$, with the upper bound depending also on $\mathrm{diam}(N)$. The result should still be interesting, but the inequality is not sharp anymore. It reads as follows:
\begin{theorem}
Let $ \kappa>0$ and $\left(N^{n}, g\right)$ be a closed Riemannian manifold with $0 \le \mathrm{Ric}_{k} \le \kappa$, then
$$
|N| \le V_k\left(\mathrm{diam}(N)\right)-\frac{\overline{R_{k}}}{n \kappa} \int_{0}^{\mathrm{diam}(N)}\left(1-e^{-\kappa \sigma_{k}(t)}\right) A_{k}(t) d t.
$$
\end{theorem}

\subsection{Effect on average mean curvature}\label{mean curvature}

There is another interpretation of Theorem \ref{thm monotone}. On one hand, we can regard it as a geometric quantity which is monotone decreasing in $r$ if $\overline {R_k}\ge0$. On the other hand, even if we discard the assumption that $\overline {R_k}\ge0$, we can still regard it as a comparison of the average mean curvature over $SN$ to the mean curvature of the geodesic sphere in the space form of curvature $k$. Indeed, as $H(p, r, \theta)=\frac{\partial }{\partial r}
\log F(p, r, \theta)$, by differentiating into the integral in \eqref{ineq monotone}, we can see that Theorem \ref{thm monotone} is equivalent to the following result.
\begin{theorem}\label{thm H}
Let $(N^n, g)$ be a closed Riemannian manifold. Suppose $r>0$ is smaller than the injectivity radius of $N$, with $r<\frac{\pi}{\sqrt{k}}$ if $k>0$, then the average over $SN$ of the mean curvature of the geodesic spheres of radius $r$ in $N$ satisfies
\begin{equation*}
\begin{aligned}
&\fint_{S N} H(p, r, \theta) d \mu_{S N}(p, \theta) \le
\frac{F_{k}^{\prime}(r)}{F_{k}(r)}-\frac{\overline { {R}_{k}}}{n }\phi_k(r)
\end{aligned}
\end{equation*}
where $\phi_k$ is a non-negative function given by \eqref{phi}.
The equality holds if and only if $(N, g)$ has constant sectional curvature $k$.
\end{theorem}

We remark that the fiber integral $\fint_{S_pN}H(p, r, \theta)d\theta$ is not the same as the average of the mean curvature over the geodesic sphere $S_p(r)$, as the measure on the geodesic sphere may not be a scaling of the spherical measure.

Combining this result with the Gauss-Bonnet theorem, we have the following corollary.
\begin{corollary}
Let $N^2$ be either a $2$-torus or $2$-sphere.
For any metric $g$ on $N$ and for any $r$ less than the injectivity radius of $(N, g)$, there exists $p\in N$ so that the average value over $\mathbb S^1$ of the mean curvature of $S_p(r)$ satisfies
$\fint_{\mathbb S^1}H(p, r, \theta)d\theta\le\frac{1}{r}$.
\end{corollary}

Gray and Vanhecke \cite[Corollary 12.7]{gray1979Riemannian} proved that on any Riemannian manifold $(N, g)$, if
$$\fint_{\mathbb S^{n-1}}H(p, r, \theta)d\theta= \frac{n-1}{r} $$
for all $p \in N$ and all sufficiently small $r$, then $N$ is flat.
As an application of Theorem \ref{thm H}, we can easily prove the following averaging version of this kind of result.
\begin{corollary}
Let $\left(N^{n}, g\right)$ be a closed Riemannian manifold whose average scalar curvature is at least $n(n-1)k$. Suppose the average over $S N$ of the mean curvature of the geodesic spheres of radius $r$ in $N$ satisfies
$$
\fint_{S N} H(p, r, \theta) d \mu_{S N}(p, \theta) = \frac{F_{k}^{\prime}(r)}{F_{k}(r)}
$$
for some $r>0$ which is smaller than the injectivity radius of $N$, with $r<\frac{\pi}{\sqrt{k}}$ if $k>0$.
Then $(N, g)$ has constant sectional curvature $k$.
\end{corollary}

\end{document}